\documentclass[10pt,oneside,reqno]{amsart}

  \usepackage[utf8]{inputenc}
  \usepackage[T1]{fontenc}
  \usepackage{lmodern}

  \usepackage{hyperref}

  \usepackage{amsmath}
  \usepackage{amsthm}
  \usepackage{amsfonts}
  \usepackage{amssymb}
  \usepackage{amscd}
  \usepackage{amsbsy}

  \usepackage{pdfpages}
  \usepackage{fancyhdr}
  \usepackage{geometry}
  \usepackage{setspace}
  \usepackage{xcolor}
  \usepackage{pifont}

  \usepackage{graphicx}
  \usepackage{tabularx}
  \usepackage{epsfig}
  \usepackage[all]{xy}
  \usepackage{tikz}
  \usetikzlibrary{matrix}

  \usepackage{fixmath}

  \usepackage{appendix}

  \usepackage{enumerate}

  \usepackage[sort, numbers]{natbib}
  \usepackage{bibentry}
  \newtheorem{theorem}{Theorem}[section]

  \newtheorem{lemma}[theorem]{Lemma}

  \theoremstyle{definition}

  \numberwithin{equation}{section}

  \newcommand{\R}{{\mathbb R}}

  \allowdisplaybreaks

\title{minimal Piecewise Linear Cones in $\mathbb{R}^4$}

\author[A. Valfells]{Asgeir Valfells}
\address{Department of Mathematics, Rice University, Houston, TX~77005, USA}
\email{Asgeir@rice.edu}
\thanks{}

\begin{document}
\begin{abstract}
	 We consider three dimensional piecewise linear cones in $\mathbb{R}^4$ that are mass minimal w.r.t. Lipschitz maps in the sense of \cite{almgren1976existence} as in \cite{Taylor76}. There are three that arise naturally by taking products of $\mathbb{R}$ with lower dimensional cases and earlier literature has demonstrated the existence of two with 0-dimensional singularities. We classify all possible candidates and demonstrate that there are no p.l. minimizers outside these five. 
\end{abstract}

\maketitle
\section{Introduction}

Note that after this work was essentially completed we learned that the main result was known and communicated by K. Brakke in 1993. While we have seen a reference to this manuscript we have not found it anywhere in the available literature or through personal correspondence.\\

In her 1976 analysis of the singularities of of two-dimensional $M(\epsilon,\delta)$ sets in $\mathbb{R}^3$ Jean Taylor classified all possible tangent cones at these singularities. These minimal cones were $\R^2$, the cartesian product of $\R$ and three half-lines meeting at $120^{\circ}$ angles (the $\mathbb{Y}$-singularity), and the cone over the 1-skeleton of a regular tetrahedron (the $\mathbb{T}$-singularity).\\

An ambitious goal, one that lies beyond our grasp, would be a similar complete analysis for three-dimensional minimal sets in $\R^4$, specifically classifying the minimal cones. The analogous strategy to that used to classify candidates in lower dimensions is to consider their intersections with $S^3$. From either the monotonicity lemma for minimal sets \cite{Taylor76} or direct calculation we can see that the tangent cones at each point of this intersection must be minimal. The classification of the minimal two dimensional cones tells us the singular points of this intersection are either the $\mathbb{Y}$-singularity or the $\mathbb{T}$-singularity. An argument of from \cite{Fleming1962} then tells us that the regular points must be locally minimal w.r.t. $S^3$ away from the singular set. \\

Considering all candidates in this usual manner would call for classifying all minimal surfaces in $S^3$ admitting $\mathbb{Y}$-type and $\mathbb{T}$-type singularities. This is too tall an order since even regular minimal surfaces in $S^3$ are as of yet unclassified. Restricting ourselves to piecewise linear cones in $\R^4$ we have a much more tenable problem. Note that this restriction arises for free in $\R^3$ since sections of great circles in $S^2$ correspond to sections of planes. \\

Throughout this paper we will show (culminating in theorem 3.4) that the only minimal piecewise linear cones in $\mathbb{R}^4$, up to rotation, are: $\mathbb{R}^3 \times 0$, $\R^2 \times \mathbb{Y}$, $\R \times \mathbb{T}$, the cone over the $2$-skeleton of the $4$-simplex, and the cone over the $2$-skeleton of the hypercube.

\section{Classifying the candidates}
The three dimensional cones are determined by their intersections with $S^3$ so for now we can restrict our focus to this intersection. Furthermore we are only considering p.l. cones so this intersection will be comprised of spherical polygons that partition $S^3$ into 3-cells. If we call the zero-dimensional singularities in the intersection vertices and the one-dimensional singularities be edges we are able to say that on the boundary of these 3-cells the vertices must have precisely 3 adjacent edges meeting at $\text{arccos}(-1/3)$ (hereafter $\alpha$) and the faces meet at dihedral angle $2\pi/3$. For the rest of this section all polygons will be isogonal spherical polygons with angle $\alpha$. Following the strategy of Heppes almost exactly \cite{Heppes64} see that up to rotation and translation there are at most 11 types of 3-cells that can occur. Call these the admissible 3-cells.\\ \\
($C_1$) Half of $S^3$, with a single face $S^2 \times {0}$.\\
($C_2$) Two sections of $S^2$ meeting at a circle.\\
($C_3$) A trihedron, comprised of 3 bigons.\\
($C_4$) A regular tetrahedron comprised of 4 regular triangles.\\
($C_5$) A triangular prism, comprised of 2 regular triangles and 3 rectangles of appropriate side lengths.\\
($C_6$) A cube, comprised of 6 squares.\\
($C_7$) A pentagonal prism, comprised of 2 regular pentagons and 5 rectangles w. appropriate side lengths.\\
($C_8$) A regular dodecahedron, comprised of 12 regular pentagons.\\
($C_9$) A decahedron comprised of 2 squares and 8 symmetric pentagons with base length equal to that of the square. Each square is adjacent to 4 distinct pentagons, two bands of pentagons separate the squares. The 1-skeleton is illustrated in fig. 10 of \cite{Taylor76}.\\
($C_{10}$) A nonahedron comprised of 3 squares and 6 symmetric pentagons whose non-adjacent equal sides have length equal to that of the square. Each pentagon is adjacent to squares on all sides of appropriate length. The 1-skeleton is illustrated in fig. 12 of \cite{Taylor76}.\\
($C_{11}$) An octahedron comprised of 4 rectangles and 4 symmetric pentagons. Each rectangle is adjacent to one other rectangle, the pentagons form a band separating the pairs of rectangles. The 1-skeleton is illustrated in fig. 11 of \cite{Taylor76}.\\ \\

When finding all partitions of $S^3$ into these cells it is useful to note that since any faces must meet at dihedral angle $2\pi/3$ any two cells that are adjacent to one another must share a face. Thankfully the faces on each of these 3-cells is uniquely determined so we will be able to distinguish many faces from one another and thus determine along which faces two admissible 3-cells can be adjacent. Clearly the regular faces are distinct from the non-regular faces, so it remains to show that the non-regular faces are in fact distinct. 

\begin{lemma}
The non-regular faces of admissible 3-cells are distinct.
\end{lemma}
\begin{proof}
Throughout this proof we will use helpful spherical geometry facts proved in a paper by Heppes \cite{Heppes64}. First we will prove that the rectangles in $C_5,C_7,C_{11}$ (hereafter $r_{.}$) are distinct from one another. We'll denote the length of the regular n-gon as $a_n$.\\ \\
(i) $r_{11}$ is non-regular|If $r_{11}$ has all sides of length $a_4$ then the pentagonal faces of $C_{11}$ have three sides of length $a_4$. However we know that a pentagon with three equal sides must be regular and cannot have edge length $a_4$. \\ \\
(ii) $r_5 \neq r_7 |$ Take a spherical rectangle w. angle $\alpha$ edge lengths $a_5$ and $a_3$, extending the edges and applying the cosine law to the protruding triangle gives $\cos(\pi-\alpha)=\tan(a_5/2)\tan(a_3/2)$. A direct calculation shows that this does not hold.\\ \\
(iii)$r_5 \neq r_{11}|$ If $r_{11}$ have a side of length $a_5$ then $C_{11}$ must have a symmetric pentagonal face with a side of length $a_5$. We know that a symmetric pentagon is determined uniquely by a single given side-length so the pentagon in question is the regular pentagon. From there we see that all the pentagonal faces in $C_{11}$ are regular pentagons so $r_{11}$ has all sides of length $a_5$ which cannot hold.\\ \\

Next we show the pentagons in $C_9$, $C_{10}$ and $C_{11}$ (hereafter $p_.$) are distinct. \\ \\
(v)$p_9 \neq p_{10}|$ A pentagon with three equal sides is regular and therefore cannot have side-length $a_4$.\\ \\
(vi)$p_9 \neq p_{11}|$ We know $r_{11}$ doesn't have side-length $a_4$ so $p_{11}$ cannot have a base of length $a_4$.\\ \\
(vii)$p_{10} \neq p_{11}|$ The longest sides of $p_{10}$ are the two of length $a_4$. Were $p_{10}=p_{11}$ then $r_{11}$ would have all sides shorter than $a_4$, which cannot occur.
\end{proof}
We've now identified each of the faces on these geometrically rigid admissible 3-cells, so we can continue on to partition $S^3$. In contrast with the analogous partition of $S^2$ into 2-cells, where there exist entire families of admissible cells we have few options locally and will follow a straight path in constructing these partitions.\\

There are three types of 3-cells that can only be adjacent to copies of themselves, $C_1$, $C_2$, and $C_3$. These partition $S^3$, respectively, into two, three, and four cells and we quickly see that the induced cones (that we get by taking the cone over the boundary of the 3-cells) are products of lower dimensional singularities with Euclidian space. Respectively these cones are, up to rotation, $ \R^3 \times 0$, $\R^2 \times \mathbb{Y}$, and $\R \times \mathbb{T}$. Constructing the other partitions is a slightly more delicate matter since it is not predetermined which 3-cell lay adjacent to another.\\ 

A helpful tool we'll use in this next lemma is the dual graph to a partition. This will be a graph where each 3-cell will be represented by a colored (depending on type) vertex with an edge between vertices that correspond to 3-cells that are adjacent along an edge. There is a richer structure available by taking a true dual polytope to the convex hull of the vertices, but we do not need it.

\begin{lemma}
There are at most 9 partitions of $S^3$ into admissible 3-cells.
\end{lemma}
\begin{proof}
We'll consider all the partitions in order of highest $C_n$ used. If $n\leq 3$ then as discussed only one type of 3-cells is used. Call the corresponding cones $T_1$, $T_2$, $T_3$.\\ \\
(i)$n=4$| The partition must be made up of only tetrahedrons. The ego-graph at any vertex of the dual graph is therefore $K_4$ so the dual graph is $K_5$. The partition is therefore into 5 copies of $C_4$. The boundaries of the 3-cells in this partition are the two-skeleton of the regular 4-simplex. Call the induced cone $T_4$.\\ \\
(ii)$n=5$| A cell of type $C_5$ must be adjacent to another along the rectangular face, but cannot be adjacent to another along the triangular face since no admissible cell has two faces of type $r_5$ adjacent along the long edge. Letting red vertices correspond to copies of $C_5$ and blue to copies of $C_4$ we see that the ego graph at red vertices is a triangular bipyramid with a red base and blue vertices away from the base. This in turn induces a dual graph with a red $K_4$ and two blue vertices adjacent to every vertex but each other. Geometrically this partition corresponds to a the 2-skeleton of a simplicial prism. Call the corresponding cone $T_5$.\\ \\
(iii)$n=6$| The only 3-cells with square faces are $C_6$. The ego graph at every vertex is $\overline{3K_2}$ and direct inspection gives us the dual graph $\overline{4K_2}$. Geometrically this partition corresponds to the 2-skeleton of a hypercube. Call the corresponding cone $T_6$.\\ \\
(iv) $n=7$|$C_7$ is the only 3-cell with pentagonal faces and having two copies adjacent to one another leads to an immediate geometric obstruction, no 3-cell has two faces of $r_7$ adjacent to one another along the short edge. There are no partitions added here.
(v)$n=8$| If there is at least one 3-cell of type $C_7$ then a similar argument to case (ii) induces partition with two copies of $C_8$ and twelve of $C_7$, the 2-skeleton of a dodecahedral prism. Call the corresponding cone $T_7$.\\ \\
(vi)$n=8$| If there are no 3-cells of type $C_7$ then all cells are of type $C_8$ and the ego graph at each vertex is the icosahedral graph. This tells us that there are three candidates for the dual graph \cite{Blokhuis85}, the point graph of the 600-cell or quotients of it w. 40 or 60 vertices. From (v) we can discern an upper bound $Vol(C_8)$. \\

First we take the isogonal pentagon w. angle $\alpha$. The spherical cosine rule gives us that $a_5=\arccos(\frac{3\cos(2\pi/5)+1}{2})\approx 0.27092$. We can then calculate the long side of the rectangle in $C_7$, b, by applying the spherical cosine law to a protruding triangle, as in $2.1.ii$, giving $b=2\arctan(\frac{\cos(\pi-\alpha)}{\tan(a_5/2)})\approx 2.3653$.
Partition $T_7$ tells us that the diameter of $C_8$ is $\pi-b \approx 0.77631$. We can  bound the volume of $C_8$ from above by $\frac{4}{3}\pi (\frac{\pi-b}{2})^3$. Taken with the surface volume of $S^3$ we get:
$$\text{Vol}(C_8) < \frac{4}{3}\pi (\frac{\pi-b}{2})^3 \approx 0.2447 < \frac{2\pi^2}{60} = \frac{\text{Vol}(S^3)}{60} $$
From here we can conclude that the only possible partition is into 120 cells of type $C_8$. Call the corresponding cone $T_8$.
\\
(vii)$n=9$| First we notice that if two 3-cells of type $C_6$ are adjacent along a face then at each edge of that face they must both be adjacent to another cell of type $C_6$. Inductively we see the only partition with copies of $C_6$ adjacent along a face is $T_6$. If two copies of $C_9$, $M_1$ and $M_2$ are adjacent along a square face then they must both be adjacent to some $M_3$ along a face of type $p_9$ which shares an edge with the respective square faces. From here we see $Q_3$ must have two faces of type $f_9$ which share an edge along their bases, but no such admissible 3-cell exists. Therefore we know that $C_9$ must be adjacent to $C_6$ along the square faces.\\ 

Assume we have a partition with a 3-cell of type $C_9$ and pick a cell of type of $C_6$ as $M_1$. We pick three faces of $M_1$ that share a vertex $V$ and affix three copies of $C_9$, $M_2-M_4$, to those faces. We know that $M_2-M_4$ are pairwise adjacent on pentagonal faces that  all share an edge $E$ with endpoints $V$ and $V'$. As four 3-cells meet at each vertex we know some $M_5$ has a vertex at $V'$. $M_5$ must meet each of $M_2-M_4$ along a pentagonal face of type $p_9$, with the top vertex of each pentagon at $V'$. That is to say $M_5$ must be a polyhedron with three pentagonal faces of type $p_9$ meeting at their top points. No such admissible polyhedron exists and thus no partitions with 3-cells of type $C_9$. \\ \\
(viii)$n=10$|By the same argument as (vii) any partition with 3-cells of type $C_{10}$ must be composed with only them and 3-cells of type $C_6$, and no two cells of the same type adjacent along a square face. Letting red vertices represent 3-cells of type $C_6$ and blue vertices represent copies of $C_{10}$ the ego graph at a red vertex is an all blue $\overline{3K_2}$ while the ego graph at a blue vertex is as illustrated in figure 1.\\ \begin{figure}[h]
\caption{The ego graph at a vertex corresponding to $C_{10}$ with vertices labelled}
\includegraphics[width=8cm]{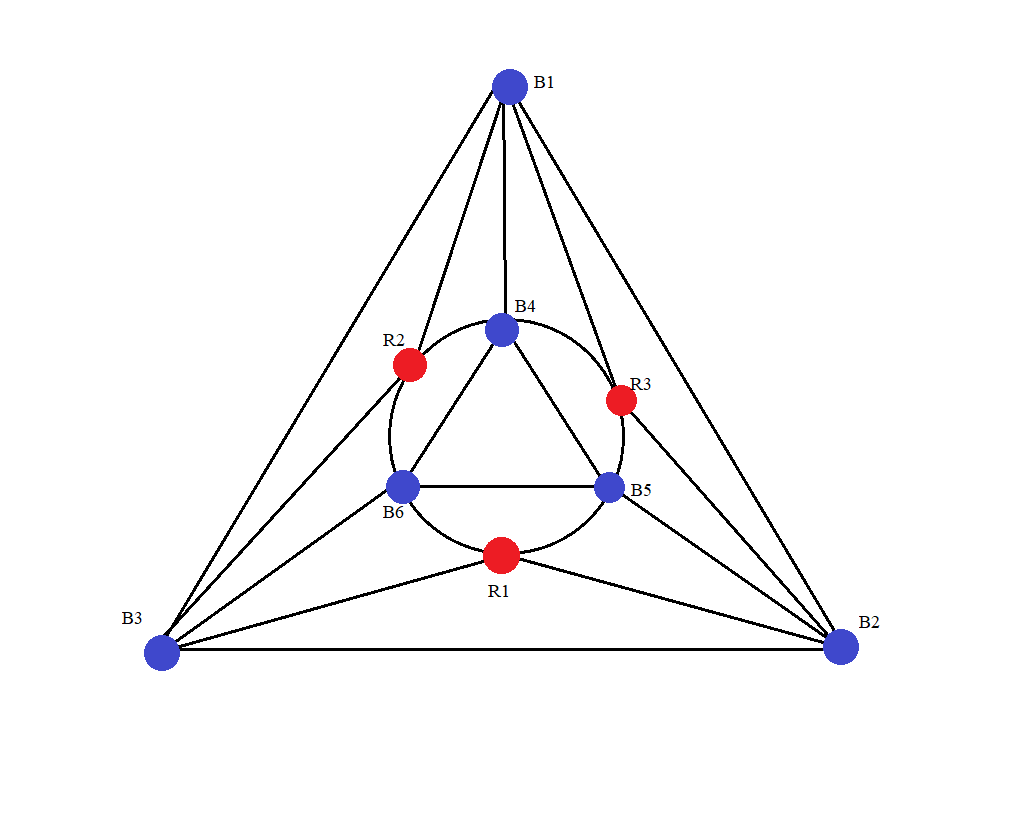}
\centering
\end{figure} To find the possible dual graph of such a partition we will construct a graph w. the prescribed ego graphs. Start by taking the ego graph at one blue vertex. Label the vertices as illustrated and the implicit blue vertex as $B_0$. From figure 1 we note that all blue vertices must have three red neighbors and any two blue vertices must share precisely two red neighbors. We conclude that there are two red vertices, $R_4$ and $R_5$ that are adjacent to $B_{1-3}$ and $B_{4-6}$ respectively. We also know that $R_4$ must be adjacent to two distinct blue 3-cliques and cannot be adjacent to $B_{4-6}$ so the graph must include three other blue vertices, $B_{7-9}$ all adjacent to one another and and $R_4$. Since the ego graph at a red vertex (specifically at $R_4$) is $\overline{3K_2}$ we find $B_7$ must be adjacent to $B_2$ and $B_3$, $B_8$ adjacent to $B_1$ and $B_3$, and $B_9$ adjacent to $B_1$ and $B_2$.\\

Next we take into account that adjacent blue vertices must share three blue neighbors and all edges of $B_1$ are accounted for so $B_4$ must also be adjacent to $B_8$ and $B_9$. This argument can then be applied to $B_5$ and $B_6$. Considering the ego graph of $R_5$ we then see that $R_5$ must be adjacent to $B_{7-9}$. This induced graph is therefore the only one with the local structure prescribed.\\

Geometrically this corresponds to five equidistributed 3-cells of type $C_6$ and ten of type $C_{10}$, with four of the latter meeting at points opposite the centers of the former. Call the corresponding cone $T_9$.\\

(ix)$n=11$| Consider the edge two rectangular faces of $C_{11}$ for any three copies of $C_{11}$ that meet at that edge three pentagonal faces meet at either vertex, but no 3-cell has 3 copies of $p_{11}$ meeting at one vertex so no tiling including $C_{11}$ exists.
\end{proof}

\section{Eliminating candidates}
With our candidates classified we can move onto the final step of discerning whether they are indeed minimal sets. We can show the cones $T_1$, $T_2$, and $T_3$ to be minimizers either by direct computation or by paired calibrations. The cone $T_4$ was shown to be a minimizer by Lawlor and Morgan \cite{LawlorMorgan94} using a paired calibration and the cone $T_6$ was shown to be a minimizer by Brakke \cite{Brakke1991} using more delicate, yet similar, methods. By direct comparison we will show that the other four candidates are not minimizers.\\ \\

The operation we will make most use of is what is described as a pop in Brakke's surface evolver. Restricting ourselves to the convex hull, $H_.$, of the vertices of a partition corresponding to cone $T_.$. Take a point $p$ in $\frac{1}{2}H_.$ away from $T_.$ and define $\phi:\mathbb{R}^4 \rightarrow \mathbb{R}^4$ s.t. if $x\in \frac{1}{2}H_.$ then $\phi(x)$ is the intersection of $\delta\frac{1}{2}H_.$ with the half-line from $p$ to $x$ and the identity otherwise. Clearly for fixed $T,H,p$ then $\phi$ is Lipschitz.

\begin{lemma}
$T_5$ is not a minimal set.
\end{lemma}
\begin{proof}
First we note that $Vol(T_5 \cap H_5) \approx 2.062$. We choose $p$ s.t. the half-line from $0$ to $p$ intersects with a cell of type $C_4$ lets us apply the pop function $\phi$. Choosing coordinates s.t. the two 3-cells of type $C_4$ are centred at $(0,0,0,\pm1)$ we enter $\phi(T_5)\cap H_5$ into the surface evolver \cite{BrakkeSurface} and to determine its volume is circa $2.133$, applying the 'go' function in the surface evolver 250 times gives us a smaller candidate with volume circa $1.98$. By direct comparison $T_5$ is not minimal.

\end{proof}
\begin{lemma}
$T_7$ is not a minimal set.
\end{lemma}
\begin{proof}
First we note that $Vol(T_7 \cap H_7) \approx 2.745$. We choose $p$ s.t. the half-line from $0$ to $p$ intersects with a cell of type $C_7$ lets us apply the pop function $\phi$. Choosing coordinates s.t. the two 3-cells of type $C_8$ are centred at $(0,0,0,\pm1)$ we enter $\phi(T_7)\cap H_7$ into the surface evolver \cite{BrakkeSurface} and to determine its volume is circa $2.759$, applying the 'go' function in the surface evolver 200 times gives us a smaller candidate with volume circa $2.671$. By direct comparison $T_7$ is not minimal.
\end{proof}
\begin{lemma}
$T_8$ is not a minimal set.
\end{lemma}
\begin{proof}
Again we will construct a direct competitor, but in a slightly different manner. Instead of defining $H_8$ in the usual manner, let it be $2B^4$. Pick any appropriate $p$ and apply the pop map $\phi$. 

On $B^4$ we replaced $720$ cones over solid spherical pentagons with $S^3$, less one cell of type $C_8$. On net we decreased the volume of the set by $$240(5\alpha-3\pi)-\frac{119}{60}\pi^2>11$$
 
\end{proof}

Finally a direct calculation in the surface evolver \cite{BrakkeSurface} shows $T_9$ is not minimal. Since no maneuvers were needed beyond refining and iterating 'go' it cannot be that $T_9$ has local the minimal structure we wanted. We suspect that there is a geometric obstruction either to constructing $C_{10}$ or partition $T_9$.

\begin{theorem}
There are precisely five piecewise linear three-dimensional minimal cones in $\mathbb{R}^4$. They are the three induced by taking the product of $\mathbb{R}$ with minimal cones in $\mathbb{R}^3$ and two with zero-dimensional singularity. They are the following:\\
    (1) The cone over the 2-skeleton of the 4-simplex.\\
    (2) The cone over the 2-skeleton of the hypercube.
\end{theorem}
\begin{proof}
In the last section we classified all possible candidates and in this section we eliminated all but five candidates, which have been demonstrated to be minimal by earlier results. \cite{Brakke1991}\cite{LawlorMorgan94}
\end{proof}

\bibliographystyle{amsplain}
\bibliography{lit}
\end{document}